\documentclass{amsart}
\usepackage{amssymb,amsmath, amsthm,latexsym}
\usepackage{mathtools,tikz-cd}
\usepackage{graphics}
\usepackage{psfrag}
\usepackage{amscd}
\usepackage{graphicx}
\newcommand{\cal}[1]{\mathcal{#1}}
\theoremstyle{plain}
\newtheorem{theo}{Theorem}

\newtheorem{lemma}{Lemma}[section]
  
\newtheorem{proposition}[lemma]{Proposition}
\newtheorem{corollary}[lemma]{Corollary}
\theoremstyle{definition}

\newtheorem{definition}[lemma]{Definition}
\newtheorem{remark}[lemma]{Remark}
\let\egthree=\phi
\let\phi=\varphi
\let\varphi=\egthree

\begin{document}
\title{Spotted disk and sphere graphs I}
\author{Ursula Hamenst\"adt}
\thanks{Partially supported by ERC Grant ``Moduli''
}
\date{March 18, 2021}

%\maketitle

\begin{abstract}
The disk graph of 
a handlebody $H$ of genus $g\geq 2$
with $m\geq 0$ marked points on the 
boundary is the graph whose vertices are isotopy classes of disks
disjoint from the marked points 
and where two vertices are connected by 
an edge of length one if they can be realized disjointly.
We show that for $m=1$ the disk graph contains
quasi-isometrically embedded copies of $\mathbb{R}^2$.
%For $m=2$ the disk graph contains for 
%every $n\geq 1$ a quasi-isometrically embedded copy of 
%$\mathbb{R}^n$. 
The same holds true for sphere graphs of the
doubled handlebody with one marked points provided that $g$ is even.
\end{abstract}

\maketitle

%\setcounter{tocdepth}{10}
%\tableofcontents

 \section{Introduction}

The \emph{curve graph} ${\cal C\cal G}$
of an oriented surface $S$ of genus $g\geq 0$
with $m\geq 0$ punctures and
$3g-3+m\geq 2$ is the graph whose vertices 
are isotopy classes of 
essential (that is, non-contractible and not homotopic 
into a puncture) simple
closed curves on $S$. Two such 
curves are connected by an edge of length one if and only
if they can be realized disjointly. 
The curve graph 
is a locally infinite hyperbolic geodesic metric space of 
infinite diameter \cite{MM99}.

A handlebody of genus $g\geq 1$  
is a compact three-dimensional manifold $H$  which can
be realized as a closed regular neighborhood in $\mathbb{R}^3$
of an embedded bouquet of $g$ circles. Its boundary
$\partial H$ is an oriented surface of genus $g$. We allow
that $\partial H$ is equipped with $m\geq 0$ marked points
(punctures) which we call \emph{spots} in the sequel.
The group ${\rm Map}(H)$ 
of all isotopy classes of orientation preserving
homeomorphisms of $H$ which fix each of the spots
is called the \emph{handlebody group} of $H$. 
The restriction of an element of ${\rm Map}(H)$ to the 
boundary $\partial H$ defines an embedding of 
${\rm Map}(H)$ into the mapping class group of $\partial H$, viewed as
a surface with punctures \cite{S77,Wa98}. 

An \emph{essential disk} in $H$ is a properly embedded
disk $(D,\partial D)\subset (H,\partial H)$ whose
boundary $\partial D$ 
is an essential simple closed curve in $\partial H$, viewed as a surface with 
punctures. An isotopy of 
such a disk is supposed to consist of such disks.

The \emph{disk graph} ${\cal D\cal G}$
of $H$ is the graph whose vertices
are isotopy classes of essential disks in $H$. 
Two such disks are connected by an edge 
of length one if and only if
they can be realized disjointly.

%A metric space $X$ is said to have \emph{asymptotic dimension}
%${\rm asdim}(X)\leq n$ if for every $R>0$ there exists a
%covering of $X$ by uniformly bounded subsets of $X$ so that
%any ball of radius $R$ intersects at most $n+1$ sets from the covering. 
%The asymptotic dimension of a curve graph is finite \cite{BF08}.

In \cite{MS10,H19a,H16} the following is shown.

\begin{theo}\label{disk}
The disk graph of a handlebody of genus $g\geq 2$ without spots is
hyperbolic.
% and has finite asymptotic dimension. 
\end{theo}

The main goal of this article
is to show that in contrast to the case of curve graphs, 
Theorem \ref{disk} is not true 
if we allow spots on the boundary.

\begin{theo}\label{diskgraph}
Let $H$ be a handlebody of genus $g\geq 2$ with one spot.
%\begin{enumerate}
%\item For $m=1$ 
Then the disk graph of $H$ 
contains quasi-isometrically embedded copies of 
$\mathbb{R}^2$. In particular, it is not hyperbolic.
%\item 
%For $m=2$  and $g\geq 3$, the disk graph of $H$ 
%contains for every $n\geq 1$ a quasi-isometrically embedded
%copy of $\mathbb{R}^n$. In particular, it is not hyperbolic, and 
%its asymptotic dimension is
%infinite. 
%\end{enumerate}
\end{theo}

%Geometric graphs which are hyperbolic are either known to have
%finite asymptotic dimension, or the question is open. 
%For example, the curve graph of 
%any surface of finite type  
%\cite{BF08}
%and the disk graph of a handlebody
%without spots \cite{H19a} have finite asymptotic dimension.
%The question is open for the free factor graph of a free group and
%for the free splitting graph.  

%In the sequel we always mean by a disk an isotopy class of disks
%where as before, an isotopy is required to stay disjoint from each marked point. 
%With this terminology, a handlebody of genus one with a single spot 
%contains a single disk. We will show in Section 3 that 
%the disk graph of a handlebody of genus one with two
%spots is a tree.

%The proof of the second part of Theorem \ref{diskgraph} uses
%$m=2$ in an essential way. 
%I do not know whether or not the asymptotic dimension of the disk graph 
%of a handlebody with a single spot or with $m\geq 3$ spots
%is finite. In view of the results in \cite{H16}, it seems possible
%that finiteness holds true for all $m\not=2$.

Theorem \ref{diskgraph} implies that 
disk graphs can not be used 
effectively to obtain a geometric understanding
of the handlebody group ${\rm Map}(H)$ of a handlebody $H$ of genus
$g\geq 3$ 
paralleling the program developed by Masur and 
Minsky for the mapping class group \cite{MM00}.
The analogue of the strategy of Masur and Minsky 
would consist of
cutting a handlebody open along an embedded disk
which yields a (perhaps disconnected) handlebody
with one or two spots on the boundary and studying disk
graphs in the cut open handlebody.

A systematic study of 
groups to which the strategy laid out by Masur and Minsky can be applied 
was recently initiated by Behrstock, Hagen and Sisto \cite{BHS17}, 
and these groups are called \emph{hierarchically hyperbolic}. 
Such groups have quadratic Dehn functions, but  
%${\rm Map}(H)$ is an exponentially distorted subgroup of 
%the mapping class group of $\partial H$ \cite{HH12}, and
for $g\geq 3$ the Dehn function of ${\rm Map}(H)$ is exponential \cite{HH19}.
Hence ${\rm Map}(H)$ can not be hierarchically hyperbolic. 
However, the geometric mechanism behind an exponential
Dehn function for ${\rm Map}(H)$ 
is not detected by the failure of being hierarchically
hyperbolic in an obvious way.

Theorem \ref{diskgraph} has an analogue for geometric graphs
related to the outer automorphism group ${\rm Out}(F_g)$ 
of the free group on $g\geq 2$ generators. 
Namely, 
doubling the handlebody $H$ yields a connected sum
$M=\sharp_gS^2\times S^1$ of $g$ copies of $S^2\times S^1$
with $m$ marked points. A deep result of Laudenbach \cite{L74} shows that
${\rm Out}(F_n)$ is a cofinite quotient of the group of isotopy classes of 
homeomorphisms of $M$. 

A doubled disk is an embedded essential sphere in $M$, which is a 
sphere which is not homotopically trivial or homotopic into a marked point. 
The \emph{sphere graph} of $M$ is the graph whose
vertices are isotopy classes of embedded
essential spheres in $M$ and where two such
spheres are connected by an edge of length one if and 
only if they can be realized disjointly. 
As before, an isotopy of spheres is required to be
disjoint from the marked points.
The sphere graph of a doubled handlebody without
marked points 
is hyperbolic \cite{HM13}.

%We can also consider sphere graphs of the complement 
%of the interior of compact embedded pairwise disjoint balls
%$B_i$ $(i=1,\dots,m)$ in $M$. As before, the vertices of this graph are
%essential spheres in $M-\cup_iB_i$, and two such 
%sphere are connected by an edge of length one if they
%can be realized disjointly.

Paralleling the result in Theorem \ref{diskgraph} we have

\begin{theo}\label{spheregraph}
Let $g\geq 2$ and let $M$ be a doubled handlebody
of genus $g$ with one marked point.
%\begin{enumerate}
%\item If $m=1$ and 
If $g$ is even then 
the sphere graph of $M$ contains quasi-iso\-me\-tri\-cally embedded
copies of $\mathbb{R}^2$. In particular, it is not hyperbolic.
%\item If $m=2$ and $g\geq 3$ then 
%the sphere graph of $M$ 
%contains for every $n\geq 1$ a quasi-isometrically embedded
%copy of $\mathbb{R}^n$. In particular, it is not hyperbolic, and 
%its asymptotic dimension is infinite. 
%\end{enumerate}
\end{theo}

The argument in the proof of 
Theorem \ref{spheregraph} 
uses Theorem \ref{diskgraph} and 
a result
in \cite{HH15} which relates the sphere graph in a connected sum 
$\sharp_g S^2\times S^1$ for $g$ even 
to the arc graph of an oriented surface of genus $g/2$ with connected
non-empty 
boundary. 
A corresponding result for odd $g$ and a non-orientable surface
with a single boundary component would yield 
Theorem \ref{spheregraph} for odd $g\geq 3$, 
but at the moment, such a result is not available.

As in the case of disk graphs, 
this indicates that sphere graphs are of limited use for obtaining 
an effective geometric understanding of 
${\rm Out}(F_g)$. 
Note that as in the case of the handlebody group, for $g\geq 3$ 
the Dehn function of 
${\rm Out}(F_g)$ 
is exponential \cite{BV12,HM10}. 

In a sequel to this article \cite{H12}, it is shown that
the disk graph of a handlebody of genus $g\geq 2$
with two spots contains
quasi-isometrically embedded $\mathbb{R}^2$, and
the sphere graph of a doubled handlebody
with two spots contains quasi-isometrically embedded $\mathbb{R}^n$ for 
every $n\geq 2$.  
We conjecture that the disk graph of a handlebody $H$ 
with $m\geq 3$ spots is quasi-isometrically embedded in the
curve graph of $\partial H$.

%The first example known to us of a 
%geometric graph
%of infinite asymptotic dimension is due to Sabalka and Savchuk
%\cite{SS14}.  The vertices of this graph are 
%isotopy classes of 
%essential separating 
%spheres in $\sharp_gS^2\times S^1$.
%Two such spheres are connected 
%by an edge of length one if and only if they can be
%realized disjointly. We use the main idea in \cite{SS14} for the proof 
%of the second part of Theorem \ref{diskgraph} and 
%of Theorem \ref{spheregraph}. 

\bigskip
\noindent
{\bf Acknowledgement:} I am very grateful to the anonymous referee
of this paper for numerous and detailed comments which helped to improved the
exposition. 
%readable, starting from a first version which in retrospect looks undigestible. 

\section{Once spotted handlebodies}

The goal of this section is to 
construct quasi-isometrically embedded copies of 
$\mathbb{R}^2$ in the disk graph of a handlebody with
a single spot.

Thus let $H$ be 
a handlebody of genus
$g\geq 2$ with a single spot. 
Let $H_0$ be the handlebody obtained from 
$H$ by removing the spot and let 
\[\Phi:H\to H_0\]
be the spot removal map. The image under $\Phi$ 
of an essential (that is, not contractible or homotopic into the spot) diskbounding 
simple closed curve in $\partial H$ is an essential diskbounding
simple closed curve in $\partial H_0$.

The handlebody $H_0$ without spots  
can be realized as a fiber bundle over 
a surface $F$ with non-empty connected boundary $\partial F$
whose fiber is the closed interval
$I=[0,1]$.  Such a fiber bundle is called an \emph{$I$-bundle}. 
We summarize from Section 3 of \cite{H16} (p.381-383)
some properties of such $I$-bundles used in the sequel.

There are two different ways a handlebody $H_0$ of genus $g$ 
can arise 
as an $I$-bundle over a surface $F$ with connected boundary
$\partial F$. 
In the first case, the surface $F$ is orientable.  Then the 
genus $g$ of $H_0$ is even and the 
$I$-bundle is trivial. The genus of $F$ equals $g/2$, 
and the boundary $\partial F$ of $F$ 
defines an isotopy class of a separating simple closed curve $c$ 
on $\partial H_0$ which 
decomposes $\partial H_0$ into two surfaces 
of genus $g/2$, with a single boundary component.

If the surface $F$ is non-orientable, then $F$ is 
the orientable 
$I$-bundle over the connected sum of $g$ projective planes with a disk, and
the $I$-bundle is non-trivial. 
The boundary $\partial F$ defines the isotopy class of 
a non-separating simple closed curve
$c$ in $\partial H_0$. 
%We refer to Section 3 of \cite{H16} for more information. 
The complement of an open annulus about $c$ in $\partial H_0$ is 
the orientation cover of $F$. 

Following Definition 3.3 of \cite{H16}, 
define an \emph{$I$-bundle generator} for $H_0$ to be
an essential simple closed curve $c\subset \partial H_0$ so that
$H_0$ can be realized as an $I$-bundle over a compact
surface $F$ with connected boundary
$\partial F$ and such that $c$ is freely homotopic to 
$\partial F\subset \partial H_0$. 
The surface $F$ is then called the \emph{base} of the $I$-bundle.
%The $I$-bundle over every essential simple
%embedded arc in $F$ with endpoints on $\partial F$ is an essential
%disk in $H_0$ which intersects $c$ in precisely two points (up to isotopy).

An $I$-bundle generator
$c$ in $\partial H_0$ is \emph{diskbusting}, 
which means that it has an essential 
intersection with every disk (see \cite{MS10,H19a}). 
Namely, the base $F$ of the $I$-bundle is a deformation retract of $H_0$. 
Thus if $\gamma$ is any essential closed curve on $\partial H_0$ 
which does not intersect $c$, then $\gamma$ projects to 
an essential closed curve on $F$. 
Such a curve is not nullhomotopic in $H_0$ and hence it can not be diskbounding.

As established in \cite{MS10,H16,H19a}, $I$-bundle generators play a special
role for the geometry of the disk graph of $H_0$. Our goal is to take advantage
of this fact for the understanding of the geometry of the handlebody with one spot. 
To this end define the 
\emph{arc graph} ${\cal A}(X)$ of a compact 
surface $X$ of genus $n\geq 1$ with 
connected boundary $\partial X$ to be the graph whose
vertices are isotopy classes of 
embedded essential arcs in $X$ with endpoints on the
boundary, and isotopies are allowed to move the endpoints 
of an arc along $\partial X$. 
Two such arcs are connected by
an edge of length one if and only if they can
be realized disjointly. The arc graph ${\cal A}(X)$ of $X$ is 
hyperbolic \cite{MS10}.

For an $I$-bundle generator $c$ in $H_0$ let 
${\cal R\cal D}(c)$ be the complete subgraph of the disk graph 
${\cal D\cal G}_0$ of 
$H_0$ consisting of disks which intersect $c$ in precisely two points.
The boundary of each such disk is an $I$-bundle over 
an arc in the base $F$ of 
the $I$-bundle corresponding to $c$ (see the discussion preceding
Lemma 4.2 of \cite{H16}). 
Namely, the $I$-bundle
over an arc in $F$ with endpoints on $\partial F$ is an embedded
disk in $H_0$. On the other hand, the boundary of 
a disk in $H_0$ defines the trivial element in the
fundamental group of $H_0$. Thus if $\beta$
is a diskbounding simple closed curve in $\partial H_0$ which intersects
$c$ in precisely two points, then 
 the homotopy classes relative to $c$ of the
two components of $\beta-c$ are exchanged under the
orientation reversing involution of $H_0$ 
which exchanges the endpoints of a fiber
in the $I$-bundle. As $\beta$ has two essential intersections with
$c$, this then implies that up to homotopy, the two components of $\beta-c$ 
trace through the two different preimages of the same points in $F$.

Now two disks intersecting $c$ in precisely two points are disjoint
if and only if the corresponding arcs in $F$ are disjoint and hence
we have 

\begin{lemma}\label{iso}
The
graph ${\cal R\cal D}(c)$ is isometric to the arc graph 
${\cal A}(F)$ of $F$.
\end{lemma}

The arc graph of a surface $F$ with non-empty
boundary $\partial F$ is a complete subgraph of 
another geometrically
defined graph, the so-called \emph{arc and curve graph}. Its vertices 
are essential simple closed curves in $F$ or arcs with endpoints 
on $\partial F$, and two such arcs or curves are connected
by an edge of length one if they can be realized disjointly. 
The arc and curve graph contains the curve graph of 
$F$ as a complete subgraph, and the inclusion of the 
curve graph into the arc and curve graph is known to be
a quasi-isometry unless $F$ is a sphere with at most
three holes or a projective plane with at most three holes
(Lemma 4.1 of \cite{H16}). 
Recall that a map $\phi:X\to Y$ be tween two metric spaces $X,Y$ is 
an \emph{$L$-quasi-isometric embedding} if for all $x,y\in X$ we have
\[d(x,y)/L-L\leq d(\phi(x),\phi(y))\leq Ld(x,y)+L,\] 
and it is called an \emph{$L$-quasi-isometry} if moreover 
its image is $L$-dense, that is, for every
$y\in Y$ there exists some $x\in X$ such that $d(\phi(x),y)\leq L$.

The arc graph ${\cal A}(F)$ of $F$ is $1$-dense in the 
arc and curve graph of $F$, but 
the inclusion of ${\cal A}(F)$ into 
the arc and curve graph of $F$ is a 
quasi-isometry only if the
genus of $X$ equals one \cite{MS10} (see also
\cite{H16}).

A \emph{coarse $L$-Lipschitz retraction} of a metric space 
$(X,d)$ onto a subspace $Y$ is a coarse $L$-Lipschitz map 
$\Psi:X\to Y$ (this means that  
$d(\Psi(x),\Psi(y))\leq Ld(x,y)+L$ for some $L\geq 1$ 
and all $x,y$)
with the additional property that there exists
a number $C>0$ with 
$d(\Psi (y),y)\leq C$ for all $y\in Y$.
If $X$ is a geodesic metric space then the image
$Y$ of a coarse Lipschitz retraction is a
\emph{coarsely quasi-convex} subspace of $X$, that is, any two points 
in $Y$ can be connected by a uniform quasi-geodesic
(for the metric of $X$) which is contained in 
a uniformly bounded neighborhood of $Y$.

\begin{lemma}\label{casewithoutspot}
Let $c$ be an $I$-bundle generator of the handlebody $H_0$. 
There exists a coarse Lipschitz retraction 
$\Theta_0:{\cal D\cal G}_0\to
{\cal R\cal D}(c)$ whose restriction to ${\cal R\cal D}(c)$ is the identity.
\end{lemma}
\begin{proof} If $c$ is a \emph{separating} $I$-bundle generator, then 
the base of the $I$-bundle can be identified with
a component $F$ of $\partial H_0-c$. Note that there two 
choices for the surface $F$. One of these two choices will be 
fixed throughout this proof.

Since the boundary $\partial D$ of a disk $D$ is an embedded simple closed curve in 
$\partial H_0$ and as
$c$ is diskbusting, the intersection $\partial D\cap F$ consists of a non-empty
collection of pairwise disjoint simple arcs with endpoints on $\partial F$.
The map 
\[\Upsilon_0:{\cal D\cal G}_0\to {\cal A}(F)\] 
which associates to a disk $D$ a component of 
$\partial D\cap F$
is coarsely well defined:
Although it depends on choices, any other choice 
$\Upsilon_0^\prime$ maps a disk $D$ to an arc disjoint from
$\Upsilon_0(D)$. 
If we denote by $Q:{\cal A}(F)\to {\cal R\cal D}(c)$ the map which 
associates to an arc $\alpha$ in $F$ the $I$-bundle over $\alpha$, then
the disks $Q(\Upsilon_0(D)),Q(\Upsilon_0^\prime(D))$   
are disjoint as well. 

Furthermore, if $D,D^\prime$ are disjoint disks then 
the arcs $\Upsilon_0(D),\Upsilon_0(D^\prime)$ are disjoint and 
hence 
$d_{{\cal D\cal G}_0}(Q\Upsilon_0(D),Q\Upsilon_0(D^\prime))\leq 1$. This shows that
$Q\circ \Upsilon_0$ is coarsely one-Lipschitz. As a disk  
$D\in {\cal R \cal D}(c)$ intersects $F$ in a single arc, we have
$Q\Upsilon_0(D)=D$. 
Thus the map $Q\circ \Upsilon_0$ is indeed a coarse 
one-Lipschitz retraction which 
completes the proof of the lemma in the case
that $c$ is separating. Note however that the relation between 
the two Lipschitz retractions constructed in this way 
from the two distinct components of $\partial H_0-c$ is unclear.  

The above argument does not
extend to non-separating $I$-bundle generators in any straightforward way.
Namely, if $c$ is a non-separating $I$-bundle
generator, then although up to homotopy, a disk which intersects
$c$ in precisely two points is invariant under 
the natural orientation reversing involution 
$\Omega$ of the corresponding $I$-bundle
which exchanges the two endpoints of a fiber, 
the projection to $F$ of the boundary of some other disk 
may have self-intersections,
and hence there is no obvious 
projection of ${\cal D\cal G}_0$ onto ${\cal R\cal D}(c)$ as in the case of a 
separating $I$-bundle generator. 

Our strategy is to establish instead 
that the inclusion ${\cal R\cal D}(c)\to {\cal D\cal G}_0$ is a 
quasi-isometric embedding.
Namely, if this holds true then as ${\cal D\cal G}_0$ is hyperbolic, the subspace
${\cal R\cal D}(c)$ is
\emph{quasi-convex}, that is, there exists a constant 
 $C>0$ such that any geodesic in 
${\cal D\cal G}_0$ connecting two points in ${\cal R\cal D}(c)$ is contained in 
the $C$-neighborhood of ${\cal R\cal D}(c)$. Then a (coarsely well defined)
shortest distance projection ${\cal D\cal G}_0\to {\cal R\cal D}(c)$
is a coarsely Lipschitz retraction by hyperbolicity.

That the inclusion ${\cal R\cal D}(c)\to {\cal D\cal G}_0$ is indeed a quasi-isometric
embedding follows from 
Theorem 10.1 of \cite{MS10} (which can only be used indirectly
as the ``holes'' are not precisely specified) and, more specifically,
Corollary 4.6 and Corollary 4.7 of \cite{H16}. These formulas 
establish that the distance in the disk graph between two disks 
$D,E$ which intersect
a given $I$-bundle generator $c$ with base $F$ in precisely two points equals the
distance in ${\cal A}(F)$ between the projections of 
$\partial D$ and $\partial E$ to $F$ 
up to a uniform constant not depending on $c$. In view
of Lemma \ref{iso}, this is what we want to show.

The details are as follows. 
Construct from the disk
graph ${\cal D\cal G}_0$ of $H_0$ another graph ${\cal E\cal D\cal G}_0$
with the same vertex set by adding additional edges as follows. 
If $D,E$ are two disks in $H_0$, and if up to homotopy, 
$D,E$ are disjoint from an \emph{essential} simple closed curve in $\partial H_0$, 
that is, a simple closed curve which is not homotopic to zero, 
then we connect $D,E$ by an edge in ${\cal E\cal D\cal G}_0$. 
This graph is called the \emph{electrified disk graph} of 
$H_0$ \cite{H16}. 

Let us denote by ${\cal E\cal R\cal D}(c)$ the subgraph of ${\cal E\cal D\cal G}_0$
whose vertex set consists of all disks which intersect the non-separating 
$I$-bundle generator $c$ in precisely two points. 
Lemma 4.2 and Lemma 4.1 of \cite{H16} show that the map which associates to an arc in 
the non-orientable surface $F$ the $I$-bundle over $F$ is a uniform quasi-isometry between 
the arc and curve graph of $F$ and  
${\cal E\cal R\cal D}(c)$. 
Furthermore, by Corollary 4.6 of \cite{H16}, the inclusion ${\cal E\cal D\cal R}(c)\to  
{\cal E\cal D\cal G}_0$ is a uniform quasi-isometric embedding. 
Here uniform means with constants not depending on $c$.

Let $\zeta:[0,m]\to {\cal E\cal R\cal D}(c)$ be a geodesic. Then $\zeta$ is a uniform 
quasi-geodesic in ${\cal E\cal D\cal G}_0$. 
Define the \emph{enlargement} $\zeta_2$ 
of $\zeta$ to be the edge path in ${\cal E\cal R\cal D}(c)$ obtained from $\zeta$ 
by replacing each edge $\zeta[k,k+1]$ by an edge path 
$\zeta_2[i_k,i_{k+1}]$ with the same endpoints as follows.

If the disks $\zeta(k),\zeta(k+1)$ are disjoint, then the edge path $\zeta_2[i_k,i_{k+1}]$ just
consists of the edge connecting these two points. 
Otherwise $\zeta(k),\zeta(k+1)$ intersect, but they 
are disjoint from an essential 
simple closed curve in $\partial H_0$. 
As each disk $\zeta(j)$ is an $I$-bundles over an arc $\alpha(j)$ 
in the surface $F$, this means
that there is an essential simple closed curve $\beta\subset F$ disjoint from
both $\zeta(k),\zeta(k+1)$. We refer to Lemma 4.2 of \cite{H16} for a detailed 
explanation. 

An essential subsurface of $F$ containing $\partial F$ is a component 
of $F-\xi$ where 
$\xi$ is a collection of pairwise disjoint 
mutually not freely homotopic essential non-boundary parallel 
simple closed curves in $F$. If $\zeta(k),\zeta(k+1)$ are disjoint from an essential
simple closed curve in $F$, then the subsurface $\hat X$ of $F$ \emph{filled} by 
$\zeta(k),\zeta(k+1)$, defined to be 
the intersection of all essential subsurfaces 
of $F$ which contain $\zeta(k)$, $\zeta(k+1),\partial F$,  
is not all of $F$. 

Let $X\subset \partial H_0$ be the preimage of $\hat X$ in $\partial H_0$. 
Then $X$ is an essential subsurface of $\partial H_0$ which 
contains the boundaries of the disks $\zeta(k),\zeta(k+1)$ and is invariant under 
the orientation reversing involution $\Omega$. 
No component of its boundary is diskbounding, and it contains 
$c$ as an $I$-bundle generator. Furthermore, no essential simple closed curve in 
$X$ (here essential means non-peripheral) is disjoint from all disks with boundary in $X$.
This follows from the fact that no essential simple closed curve in $X$ is disjoint
from both $\zeta(k)$ and $\zeta(k+1)$ as $\zeta(k),\zeta(k+1)$ are 
invariant under $\Omega$ and their projection to $F$ fill the projection $\hat X$ of $X$.
A subsurface $X$ of $\partial H_0$ with these properties 
is called \emph{thick} in \cite{H16}. 

The complete subgraph ${\cal E\cal D\cal G}(X)$ of ${\cal E\cal D\cal G}_0$
whose vertex set is the 
set of all disks with boundary in $X$ is an 
electrified disk graph for $X$. 
By Corollary 4.6 of \cite{H16}, its subgraph 
${\cal E\cal R\cal D}(c,X)$ 
of all disks
which intersect $c$ in precisely two points is uniformly 
quasi-isometrically embedded in the electrified 
disk graph of $X$. Note that Corollary 4.6 of \cite{H16} only states 
that this graph is uniformly quasi-convex, however Corollary 2.8
of \cite{H16} shows that indeed, the inclusion of each of these
graphs into the electrified disk graph of $X$ is a uniform
quasi-isometric embedding. Furthermore, 
by Lemma 4.2 of \cite{H16}, the graph ${\cal E\cal R\cal D}(c,X)$ 
is 4-quasi-isometric to the arc 
and curve graph of $\hat X$ where we require arcs to have endpoints
on the distinguished boundary component $c$ of $\hat X$.

If $\hat X$ is the complement of an orientation reversing simple
closed curve disjoint from $c$, then $X$ is the complement in 
$\partial H_0$ of an essential simple closed curve. In this case we 
define $\zeta_2[i_k,i_{k+1}]$ to be 
the path in ${\cal E\cal R\cal D}(c,X)$ connecting $\zeta(k)$ to 
$\zeta(k+1)$ which consists of $I$-bundles over arcs in $\hat X$ 
defined by a geodesic in the arc and curve graph of $\hat X$.
That is, from a geodesic in the arc and curve graph of $\hat X$
we construct first an edge path of at most twice the length 
with the property that among two consecutive vertices, at least one is an arc, 
and then we view this edge path as an edge path in the graph
${\cal E\cal R\cal D}(c,X)$. By Corollary 4.6 of \cite{H16}, $\zeta_2[i_k,i_{k+1}]$ is 
a uniform quasi-geodesic in ${\cal E\cal D\cal G}(X)$.
If the complement of $\hat X$ contains an orientation preserving
simple closed curve which does not bound a M\"obius band, 
then the complement of $X$ in $\partial H_0$
contains at least two disjoint simple closed curves and we define
$\zeta_2[i_k,i_{k+1}]$ to be the edge between $\zeta(k)$ and 
$\zeta(k+1)$.

The resulting edge path $\zeta_2$ in ${\cal E\cal R\cal D}(c)$ 
has the property that  
two consecutive vertices, which are disks $D,E$ intersecting
$c$ in two points, are either disjoint, or their boundaries
lie in the same proper thick $\Omega$-invariant 
subsurface $X$ of $\partial H_0$ 
containing $c$ as an $I$-bundle generator. Moreover, 
$D,E$ are connected by an edge in the graph ${\cal E\cal R\cal D}(c,X)$.
In particular, the complement of the subsurface of 
$\partial H_0$ filled by $D,E$ contains at least two disjoint essential simple closed
curves. 

Let ${\cal E\cal D\cal G}(2,\partial H_0)$ be the graph 
whose vertex set is the set of disks and where two disks are 
connected by an edge if either they are disjoint, or if they are disjoint from 
a multicurve consisting of at least two non-homotopic components.  
By Theorem 5.5 of \cite{H16}, the graph ${\cal E\cal D\cal G}(2,\partial H_0)$ is 
hyperbolic, and it is an electrification of the disk graph of $H_0$. 
This means that it has the same vertex set as 
the disk graph of $H_0$, and it is obtained from this disk graph 
by adding edges.

Theorem 5.5 of \cite{H16} also shows that 
the path $\zeta_2$ is a uniform quasi-geodesic in ${\cal E\cal D\cal G}(2,\partial H_0)$.
Namely, following Section 5 of \cite{H16}, define a simple closed curve 
$\gamma\subset \partial H_0$ to be 
\emph{admissible} if $\gamma$ is neither
diskbounding nor diskbusting. Each such curve defines a thick subsurface of 
$\partial H_0$. Write ${\cal E\cal D\cal G}(\partial H_0-\gamma)$ to denote
the electrified disk graph of $\partial H_0-\gamma$ and let ${\cal F}(\gamma)$
to be the complete subgraph of ${\cal E\cal D\cal G}(2,\partial H_0)$ whose vertex
set consists of all disks which are disjoint from $\gamma$. A disk $D\subset {\cal F}(\gamma)$
defines a vertex in ${\cal E\cal D\cal G}(\partial H_0-\gamma)$. 

Following Section 2 of \cite{H16}, define the \emph{enlargement} of 
a uniform quasi-geodesic $\eta:[0,n]\to {\cal E\cal D\cal G}_0$ with no
backtracking as follows. Assume that $\eta(j),\eta(j+1)\in 
{\cal E\cal D\cal G}(\partial H_0-\gamma)$ for some admissible simple closed
curve $\gamma$ and some $j<n$; then replace 
the edge $\eta[j,j+1]$ by a geodesic (or uniform quasi-geodesic)
in ${\cal E\cal D\cal G}(\partial H_0-\gamma)$. Note that if $\eta(j),
\eta(j+1)$ are disjoint from an essential simple closed curve in 
$\partial H_0-\gamma$, then there is an edge between 
$\eta(j),\eta(j+1)$ in ${\cal E\cal D\cal G}(\partial H_0-\gamma)$. 
Theorem 5.5 of \cite{H16} states that enlargements of uniform quasi-geodesics
in ${\cal E\cal D\cal G}_0$ are uniform quasi-geodesics in 
${\cal E\cal D\cal G}(2,\partial H_0)$.  

Now the above construction takes as input a geodesic in 
${\cal R\cal D}(c)$ and associates to it an enlargement, chosen in 
such a way that this enlargement consists  of disks whose boundaries
intersect $c$ in precisely two points. Using once more Theorem 5.5 of 
\cite{H16}, this shows that inclusion defines a quasi-isometric embedding
of the complete subgraph of ${\cal E\cal D\cal G}(2,\partial H_0)$ of disks which 
intersect $c$ in precisely two points into the graph 
${\cal E\cal D\cal G}(2,\partial H_0)$.

This construction can be iterated. In the next step, we modify the path $\zeta_2$ to 
a path $\zeta_3$ by replacing suitable edges by edge paths as follows. 
Consider two 
consecutive vertices $\zeta_2(k),\zeta_2(k+1)$ of $\zeta_2$. These are disks which 
intersect $c$ in precisely two points. If they are not disjoint, then 
there exists an essential simple closed curve $\gamma\subset F$ which is disjoint 
from both $\zeta_2(k),\zeta_2(k+1)$. If $\gamma$ is orientation preserving and does not 
bound a M\"obious band, 
then $\gamma$ has two disjoint preimages $\gamma_1,\gamma_2$ in $\partial H_0$, and 
the complement of these preimages is an $\Omega$-invariant 
thick subsurface $X$ of $\partial H_0$ containing
$c$ as an $I$-bundle generator. 
Replace $\zeta_2[k,k+1]$ by a geodesic in ${\cal E\cal D\cal G}(X)$
with the same endpoints.  This geodesic can
be chosen to be the preimage of a geodesic in the arc and curve graph of 
$F-\gamma$. 
Proceed in the same way if the complement of the subsurface of $F$ filled by
$\zeta(k)\cap F,\zeta(k+1)\cap F$ only contains orientation reversing primitive simple closed
curves.

In finitely many steps
we construct in this way a path in the graph ${\cal R\cal D}(c)$ connecting the endpoints of 
$\zeta$. Its length roughly equals the sum of the subsurface projections of 
the projection of 
its endpoints to $F$, where the sum is over all essential subsurfaces of $F$ containing the boundary
$\partial F$.  
In particular, by the distance formula in 
Corollary 6.3 of \cite{H16}, its length does not 
exceed a uniform multiple of the distance in ${\cal D\cal G}_0$ between
its endpoints. 
%Note to this end that the length of the path equals the second sum of the formula 
%and hence as the path is contained in the disk graph and connects the given endpoints, 
%and hence its length is not smaller than the distance between its endpoints, the first sum in the 
%formula plays no role. 
The statement 
also follows as by the main result of \cite{H16}, the so-called 
hierarchy paths, constructed from a geodesic in ${\cal E\cal D\cal G}_0$ in the above inductive fashion, 
 are uniform quasi-geodesics in the disk graph. 

As a consequence, taking the $I$-bundle over an arc in $F$ defines an isometry between
the arc graph of $F$ and the graph ${\cal R\cal D}(c)$, and this graph is quasi-isometrically
embedded in ${\cal D\cal G}_0$. This is what we wanted to show.
%
%To complete the proof of the lemma we are left with showing that
%the inclusion ${\cal R\cal D}(c)\to {\cal D\cal G}_0$ is a 
%quasi-isometric embedding. This is explicitly stated in the main result
%of \cite{H19a} (see also \cite{H16}) 
%and also follows slightly indirectly from Theorem 12.1 of 
%\cite{MS13}. 
\end{proof}

Our goal is to use $I$-bundle generators in $\partial H_0$ to construct 
quasi-isometrically embedded euclidean planes in the disk graph of $H$.
In analogy to \cite{H19a}, we define
an \emph{$I$-bundle generator} for the spotted handlebody
$H$ to be a simple closed
curve in $\partial H$ whose  
image under the spot forgetful 
map $\Phi$ is an $I$-bundle generator in $\partial H_0$.

Let $(c_1,c_2)\subset \partial H$ be a 
pair of non-isotopic disjoint
$I$-bundle generators so that $\partial H-\{c_1\cup c_2\}$ has a connected component
which is an annulus containing the spot in its interior. 
Then up to isotopy, $\Phi(c_1)=\Phi(c_2)=c$
for an $I$-bundle generator $c$ in $H_0$.

The following construction is due to Kra; 
we refer to \cite{KLS09} for details and for some applications.
For its formulation, for a pair $(c_1,c_2)$ of disjoint $I$-bundle generators
on $\partial H$ as in the previous paragraph 
let ${\cal R\cal D}(c_1,c_2)$ be the 
complete subgraph of the disk graph ${\cal D\cal G}$ of $H$ 
whose vertex set consists of all disks
which intersect each of the curves $c_1,c_2$ in precisely two points.
Note that if $D\in {\cal R\cal D}(c_1,c_2)$ then the image of 
$D$ under the spot removing map $\Phi$ is 
contained in ${\cal R\cal D}(c)$ where $c=\Phi(c_i)$.

In the next lemma we denote by abuse of notation the map 
${\cal D\cal G}\to {\cal D\cal G}_0$ induced by the
spot forgetful map $\Phi$ again by $\Phi$. Furthermore,
for the remainder of this section we represent a disk
by its boundary, that is, we view the disk graph as the complete 
subgraph of the curve graph of $\partial H$ whose vertex set
is the set of diskbounding curves.

\begin{lemma}\label{isometricembedding}
Let $(c_1,c_2)$ be a pair of I-bundle generators bounding a punctured annulus
and let $c=\Phi(c_1)=\Phi(c_2)$. 
There exists a simplicial embedding $\iota:{\cal D\cal G}_0\to 
{\cal D\cal G}$ with the following properties.
\begin{enumerate}
\item $\Phi\circ \iota$ is the identity.
\item $\iota$ maps ${\cal R\cal D}(c)$ into
${\cal R\cal D}(c_1,c_2)$.
\end{enumerate}
\end{lemma}
\begin{proof}
Note first that
there is a natural orientation reversing involution $\rho_0$ of 
$\partial H_0$ which exchanges the endpoints of the fibres of the 
interval bundle over the base $F$. This involution 
fixes $c$ and 
preserves up to isotopy each diskbounding simple closed 
curve which intersects $c$ in precisely two points. We refer to the 
discussion before Lemma \ref{iso} for this fact. 

Choose a hyperbolic metric
$g_0$ on $\partial H_0$ which is invariant under $\rho_0$
and let $\hat c$ be the geodesic representative of $c$. 
This makes sense since the geodesic representative of a simple closed curve
is simple. Choose a point 
$p\in \hat c$ not contained
in any diskbounding simple closed geodesic; this is possible since 
each diskbounding simple closed geodesic intersects $\hat c$ transversely
in finitely many points and hence the set of all points
of $\hat c$ contained in a diskbounding closed geodesic is countable.
View $p$ as a marked point on 
$\partial H_0$; then the geodesic representative
of a diskbounding curve $\alpha$ in $\partial H_0$ 
is a diskbounding curve $\iota(\alpha)$ in $\partial H_0-\{p\}$.
Via identification of a disk with its boundary, this construction defines
a simplicial 
embedding 
\[\iota:{\cal D\cal G}_0\to  
{\cal D\cal G}\] with the property that $\Phi\circ \iota$
equals the identity. 
Note that $\iota$ is simplicial and hence one-Lipschitz because
the geodesic representatives of two disjoint simple closed curves
are disjoint. 
Furthermore, we clearly have 
$\iota({\cal R\cal D}(c))
\subset {\cal R\cal D}(c_1,c_2)$. 
\end{proof}

The situation in the following discussion is illustrated in Figure A.
Let $B$ be the connected component of 
$\partial H-\{c_1,c_2\}$ containing the spot (this is a once
spotted annulus). 
Let $\Lambda$ be 
a diffeomorphism of $\partial H$
which preserves the complement of $B$ 
(and hence the boundary of $B$)
pointwise and 
which pushes the spot in $B$ one full turn around a
central loop in $B$. 
The isotopy class of $\Lambda$ 
is contained
in the kernel of the homomorphism 
${\rm Mod}(\partial H)\to {\rm Mod}(\partial H_0)$ 
induced by the spot removal map $\Phi$. 
The map $\Lambda$ extends to 
a diffeomorphism of the handlebody $H$. This can be seen as
in the case of point-pushing in a surface: 
Identify the image of $B$ under the spot removal map 
$\Phi$ with a closed annulus $A$.
Choose
a neighborhood $N$ of the punctured annulus $B$ in $H$ 
which is homeomorphic to $A\times [0,1]$, 
with one interior point removed from $A\times \{0\}$. 
Gradually undo the rotation of the marked point as one moves
towards $A\times \{1\}\cup \partial A\times [0,1]$. Therefore
the diffeomorphism $\Lambda$ generates
an infinite cyclic group of simplicial isometries of 
${\cal R\cal D}(c_1,c_2)$ which we denote again by $\Lambda$. 
With this notation, $\Phi\circ \Lambda=\Phi$. 
%An orbit of this 
%group is mapped by $\Psi_1$ to the same 
%arc in $X$.
\begin{figure}[ht]
\begin{center}
%\psfrag{Figure B}{Figure A}
\includegraphics[width=0.7\textwidth]{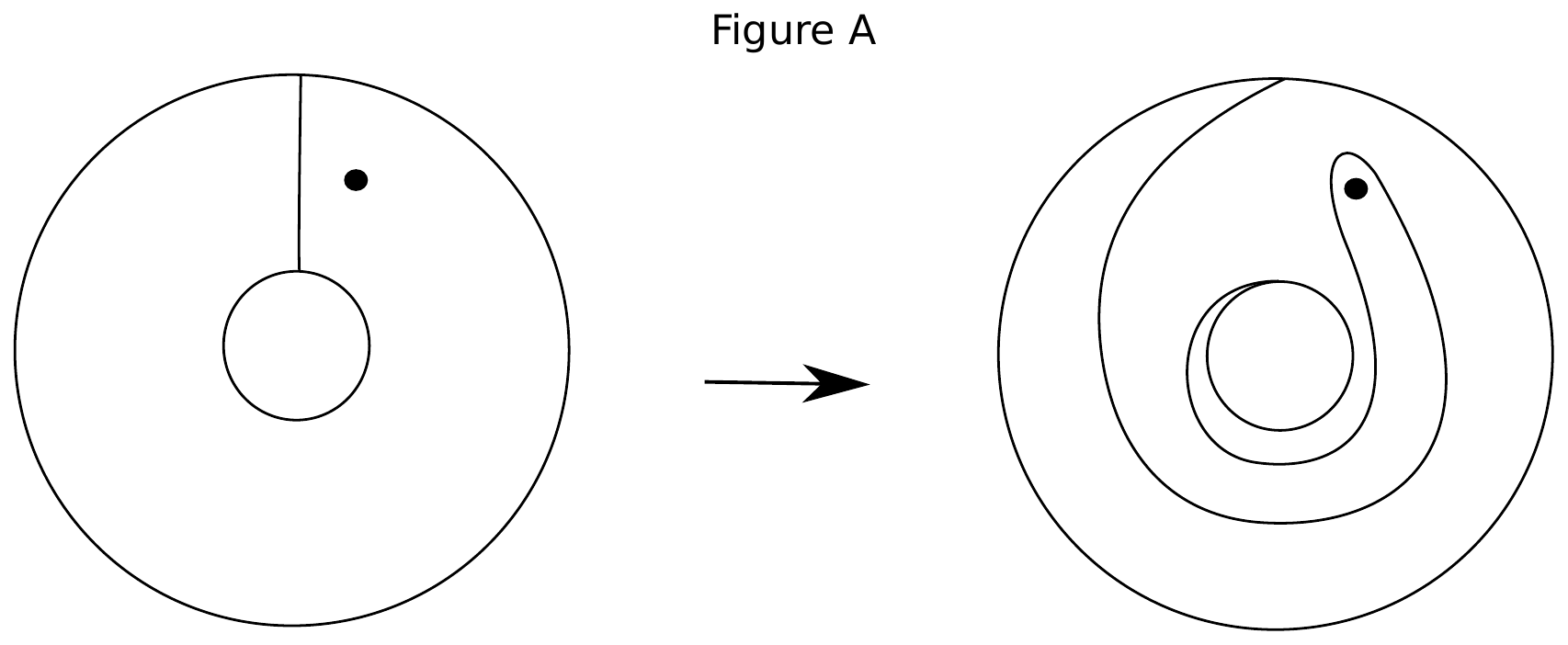} 
\end{center}
\end{figure}

Let $\Theta_0:{\cal D\cal G}_0\to {\cal R\cal D}(c)$ be
as in Lemma \ref{casewithoutspot}. 
Define
\begin{equation}\label{theta}\Theta=\Theta_0\circ \Phi:{\cal D\cal G}\to {\cal R\cal D}(c).
\end{equation}
Observe that 
$\Theta(\iota(D))=\Theta_0(D)$ for all disks $D\in {\cal D\cal G}_0$. 
This then implies that  
$\Theta(\iota(D))=D$ for all  $D\in {\cal R\cal D}(c)$. 
Furthermore, $\Theta$ is coarsely Lipschitz.
Namely, the puncture forgetful map $\Phi$ is simplicial and hence one-Lipschitz, and
the map $\Theta_0$ is a coarse Lipschitz retraction by 
Lemma \ref{casewithoutspot}. Moreover, we have
\[\Theta(\Lambda(D))=\Theta(D)\]
for all disks $D$.

Recall from Lemma \ref{iso} that 
${\cal R\cal D}(c)$ is isometric to the arc graph
${\cal A}(F)$ of $F$. 
Define a distance
$d_0$ on ${\cal R\cal D}(c)\times \mathbb{Z}$ by
\[d_0((\alpha,a),(\beta,b))=d_{{\cal R\cal D}(c)}(\alpha,\beta)+\vert a-b\vert\]
where $d_{{\cal R\cal D}(c)}$ 
denotes the distance in ${\cal R\cal D}(c)$. Let moreover
\[\Omega=\cup_k\Lambda^k\iota({\cal R\cal D}(c)),\]
equipped with the restriction of the distance function of 
${\cal D\cal G}$.

In the following lemma, the fact that the map 
$\Psi$ is well defined is part of the claim which is established
in its proof. 

\begin{lemma}\label{firststep}
The  map 
$\Psi:\Omega\to {\cal R\cal D}(c)\times \mathbb{Z}$
which maps $D\in \Lambda^k\iota ({\cal R\cal D}(c))$ to 
$\Psi(D)=(\Theta(D),k)$ is a bijective quasi-isometry.
\end{lemma}
\begin{proof}
Recall that $\Theta(D)=\Theta(\Lambda^k(D))$ for all disks $D$ 
and all $k$ and that furthermore the restriction of 
$\Theta$ to $\iota ({\cal R\cal D}(c))$ is an isometry.
In particular, if $D_0,E_0$ are distinct disks
in ${\cal R\cal D}(c)$ then
$\Theta(\iota(D_0))\not=
\Theta(\iota(E_0))$ and hence $\iota(D_0)\not=
\Lambda^k(\iota(E_0))$ for all $k$.

We claim that for every disk $D\in \Omega$ 
the following hold true.
\begin{enumerate}
\item $D\not=\Lambda^k(D)$ for all $k\not=0$.
\item If $D\in \iota({\cal R\cal D}(c))$ then 
$\Lambda^kD\not\in \iota ({\cal R\cal D}(c))$ for all $k\not=0$. 
\item The disks $D$ and $\Lambda(D)$ can be realized disjointly.
\item Two disks $D\in \Lambda^k\iota({\cal R\cal D}(c)),
E\in \Lambda^\ell\iota ({\cal R\cal D}(c))$ are disjoint only if
$\vert k-\ell\vert \leq 1$.
\end{enumerate}

To show the claim 
let $D\in \Omega$ and for $k\in \mathbb{Z}$ let 
$D_k=\Lambda^k(D)$. Figure A shows that 
for $\ell\geq 1$, the disk 
$D_{k+\ell}$ has precisely $2\ell -2$ essential 
intersections with $D_k$, and these intersection points
are up to isotopy contained in the annulus $B$. 
This yields part (3) of the above claim, and 
part (4) follows from the same argument.  
Furthermore, the
twist parameter $k$ can be recovered from the geometric 
intersection numbers
between $\Lambda^k(D)$ and $\Lambda^{-1}(D),D,\Lambda(D)$.
For example, if $k\geq 2$ then these intersection numbers
equal $2k,2k-2,2k-4$, respectively, and if 
$k\leq -2$ then these intersection numbers are
$-2k-4,-2k-2,-2k$. 
This establishes
part (1) of the above claim, and part (2) follows from part (1) and the fact that
the map $\iota$ is an embedding. In particular, 
$\Omega=\sqcup_k \Lambda^k \iota({\cal R\cal D}(c))$ (disjoint union).

As a consequence, there exists a 
map $\Psi$ as claimed in the statement of the lemma, and this map is a bijection.
Now $\Omega\subset {\cal R\cal D}(c_1,c_2)$ and 
the restriction of the map $\Theta$ to ${\cal R\cal D}(c_1,c_2)$ is just
the map induced by the spot forgetful map and hence it
is one-Lipschitz. 
Part (4) of the above claim implies that 
the map $\Psi$ is two-Lipschitz.

As $\Lambda^k\iota({\cal R\cal D}(c))$ is isometric
to ${\cal A}(F)$ for all $k$, 
the inverse of $\Psi$ 
which associates to a pair 
$(D,k)\in {\cal R\cal D}(c)\times \mathbb{Z}$ 
the disk $\Lambda^k(\iota(D))$ 
is coarsely one-Lipschitz. This shows that 
indeed, the map 
$\Psi$ is a quasi-isometry. 
\end{proof}

The following proposition is the main remaining step towards
a proof of 
Theorem \ref{diskgraph}.

\begin{proposition}\label{calculate2}
There is a coarse Lipschitz retraction 
${\cal D\cal G}\to \cup_k\Lambda^k\iota({\cal R\cal D}(c))
=\Omega$.
Moreover, 
$\Omega$ is a coarsely quasi-convex
subset of ${\cal D\cal G}$.
\end{proposition}
\begin{proof} For the construction of the Lipschitz retraction, we take 
advantage of the fact that any free homotopy class 
on a complete hyperbolic surface of finite area can be 
represented by a unique closed geodesic.

As in the proof of 
Lemma \ref{isometricembedding}, 
let $\rho_0$ be an orientation
reversing involution of $\partial H_0$ 
which fixes the $I$-bundle generator
$c$ pointwise. 
This involution determines an involution $\rho$ of 
the complement in $\partial H$ of the interior ${\rm int}(B)$ 
of the annulus $B$ which 
exchanges the curves $c_1$ and $c_2$.  
Write as before $\Omega=\cup_k\Lambda^k\iota({\cal R\cal D}(c))$.

Choose a complete finite area hyperbolic metric on
$\partial H$ (so that the marked point becomes a puncture)
with the property that the involution $\rho$ of 
$\partial H -{\rm int}(B)$ is an isometry for this metric 
which maps the 
geodesic representative $\hat c_1$ of $c_1$ to 
the geodesic representative $\hat c_2$ of $c_2$.
This metric restricts to a hyperbolic metric on 
 the once punctured annulus $B$ with 
geodesic boundary.  
We use this hyperbolic metric to determine for each pair of points
$x_i\in \hat c_i$ $(i=1,2)$ a sample arc in $B$ 
connecting these two points as follows. 

%The once punctured annulus $B$ is identified with
%the closure of $\partial H-(X\cup Y)$. 
%Every simple closed curve on $\partial H$ is represented by 
%a simple closed geodesic with respect to this metric.

%Let $\alpha\subset \partial H$ be a diskbounding simple closed geodesic.
%Then $\alpha$ intersects the geodesics
%$\hat c_i$ transversely, and $\alpha\cap B$
%is a collection of at least two 
%simple pairwise disjoint arcs. 

Choose a shortest geodesic arc $\alpha$ connecting the
two boundary components of $B$. By perhaps pulling back the 
hyperbolic metric with a diffeomorphism of $B$ which preserves
the boundary of $B$ pointwise, we may
assume that $\alpha$ is contained in the
geodesic representative of one of the curves from 
$\iota({\cal R\cal D}(c))$.
Cutting $B$ open along $\alpha$
yields a once punctured right angled
rectangle $R$ with geodesic sides, where 
two distinguished sides
come from the arc $\alpha$.
For any pair of points $x_1,x_2$ 
on the remaining two sides,
choose a shortest geodesic  arc $\alpha(x_1,x_2)$ in $R$ 
connecting these two points.
Such an arc is simple, but it may not be unique. 
By convexity, $\alpha(x_1,x_2)$ 
is disjoint from $\alpha$ if its
endpoints are disjoint from the endpoints of $\alpha$. 
Note that 
as the spot of $\partial H$ is a puncture for the hyperbolic metric,
the geodesic arcs $\alpha(x_1,x_2)$ are disjoint from the spot, and 
$\alpha(x_1,x_2)$ is not necessarily a shortest arc 
in $B$ with fixed endpoints.

This construction yields for any pair of points
$x_1\in \hat c_1,x_2\in \hat c_2$ an oriented geodesic arc 
$\alpha(x_1,x_2)\subset B$ with 
endpoints $x_1,x_2$ such that any two of 
these arcs connecting distinct pairs of points on 
$\hat c_1,\hat c_2$  
intersect in at most two points. Furthermore, 
each of these arcs intersects a geodesic representative of a curve in 
$\iota({\cal R\cal D}(c))$ in at most two points. 
%This is possible by the definition of the map $\iota$. 

The geodesic arcs $\alpha(x_1,x_2)$ serve as a base marking
to measure the twisting of a diskbounding simple closed curve
relative to a simple closed curve in the set $\iota({\cal R\cal D}(c))\subset
{\cal D\cal G}$. This is reminiscent to the definition of a twist parameter
for a simple closed curve crossing through $c$ relative to a fixed marking
of the surface $\partial H_0$. As we have to measure twisting about
the puncture, we have to take care of a pair of twist parameters
about the simple closed curves $c_1,c_2$. Our strategy to this
end is to put the intersection of a simple closed diskbounding curve
$\beta$ with $\partial H-B$ into a normal form and use this normal
form and the a priori chosen arcs $\alpha(x_1,x_2)$ to determine
a twisting datum for $\beta$.
We next construct such a normal form for the intersection of 
$\beta$ with $\partial H-B$ using hyperbolic geometry.

Thus let $\beta$ be a diskbounding simple closed curve on $\partial H$.
The intersection of $\beta$ with $\partial H-{\rm int}(B)$ consists
of a non-empty collection $\zeta$ of finitely many pairwise disjoint 
simple arcs with 
endpoints on $\hat c_1,\hat c_2$. 
Each such arc is freely homotopic 
relative to $\hat c_1,\hat c_2$ to a unique geodesic
arc which meets $\hat c_1,\hat c_2$ orthogonally at its endpoints. 
%These geodesic arcs are simple.

We claim that the components of 
the thus defined collection 
$\hat \zeta$ of geodesic arcs are pairwise disjoint.
However, some of these arcs may have nontrivial multiplicities
as $\beta\cap (\partial H-{\rm int}(B))$ may 
contain several components which are 
homotopic relative to the boundary. To verify the claim,
double each component $X$ of the hyperbolic surface $\partial H-{\rm int}(B)$ 
along its boundary. 
The resulting, possibly disconnected,
closed hyperbolic surface $S$ admits
an isometric involution $\sigma$ preserving the components of $S$ 
whose fixed point set is precisely the 
image $C$ of the boundary 
of $\partial H-{\rm int}(B)$ in the doubled manifold. 
The double of the above collection $\zeta$ of arcs is a collection of 
simple closed curves on $S$ which are invariant under $\sigma$. 

The free homotopy classes of these closed curves are $\sigma$-invariant and 
hence the same holds true for their
geodesic representatives: Namely, if $\gamma$ is the geodesic representative
of such a free homotopy class, then $\gamma$ intersects the 
geodesic multicurve $C$ 
in precisely two points. Let $\gamma_1$ be the component of 
$\gamma-C$ of smaller length. Then 
$\gamma_1\cup \sigma(\gamma_1)$ is a simple closed curve freely homotopic to 
$\gamma$, and its length is at most the length of $\gamma$. But $\gamma$ is 
the unique simple closed curve of minimal length in its free homotopy class
and hence $\gamma=\gamma_1\cup \sigma(\gamma_1)$. Thus
$\gamma$ intersects $C$ orthogonally, and 
$\gamma\cap X$ is a component
of the arc system $\hat\zeta$. 
The claim now follows from the 
well known fact that the geodesic representative of a simple closed
multicurve on a hyperbolic surface is a simple closed multicurve. 

As a consequence of the above discussion, the order of the endpoints of the
components of $\beta-{\rm int}(B)$ 
on $\hat c_1\cup \hat c_2$ coincides with the order of the endpoints of 
the collection of geodesic arcs $\hat\zeta$ which meet $\hat c_1\cup \hat c_2$ orthogonally
at their endpoints and are 
freely homotopic to the components of 
$\beta-{\rm int}(B)$. 
This implies that 
a diskbounding simple closed curve $\beta$
on $\partial H$ can be homotoped to a curve 
$\hat \beta$ of the following form.
\begin{enumerate}
\item[(i)]
The restriction of $\hat \beta$ to $\partial H-{\rm int}(B)$ consists of a finite collection of pairwise disjoint 
geodesic arcs 
which meet $\hat c_i$ orthogonally at their endpoints. Some of these
arcs may occur more than once. 
\item[(ii)]
The restriction of $\hat \beta$ to the once punctured annulus $B$ 
consists of a finite non-empty
collection of arcs connecting $\hat c_1$ to $\hat c_2$
and perhaps a finite number of arcs which go around the puncture and return
to the same boundary component of $B$.
 Distinct such 
arcs have disjoint interiors. 
\end{enumerate}
The curve $\hat \beta$ is uniquely
determined by 
$\beta$ and the choice of the hyperbolic metric on $\partial H$ 
up to a homotopy of the components of 
$\hat \beta\cap B$ with fixed 
endpoints (note that the above construction does not determine uniquely
the intersection of $\hat \beta$ with $B$). 
This completes the construction of a normal form for a diskbounding simple
closed curve $\beta$ on $\partial H$.

The goal is to use this normal form to construct a Lipschitz retraction of 
${\cal D\cal G}$ as stated in the proposition by associating to a diskbounding
simple closed curve $\beta$ in ${\cal D\cal G}$ a pair $\Psi^{-1}(\Theta(\beta),k)$
where $\Psi$ is as in Lemma \ref{firststep}, where $\Theta$ is as in (\ref{theta}) 
and where $k$ is a twist parameter, read off from the intersection of the normal
form with the once punctured annulus $B$.
We first check compatibility of this twist parameter construction with
the twist parameter stemming from the decomposition 
$\Omega=\cup_k\Lambda^k\iota({\cal R\cal D}(c))$.

By construction of the map $\iota$, if $\beta=\iota(\beta^\prime)\in \iota({\cal R\cal D}(c))$ then
$\hat \beta\cap \partial H-{\rm int}(B)$ is just the lift of the geodesic representative of 
$\beta^\prime$ to $\partial H-{\rm int}(B)$ for the following hyperbolic metric 
on $\partial H_0-c$. Recall that the metric on $\partial H$ was chosen in such a way
that there exists an orientation reversing involution $\rho$ which maps 
$\hat c_1$ to $\hat c_2$. Cutting ${\rm int}(B)$ off $\partial H$ and 
gluing $c_1$ to $c_2$ with the isometric involution $\rho$ constructs from 
$\partial H-{\rm int}(B)$ a hyperbolic
surface which can be viewed as a hyperbolic metric on $\partial H_0$. 
Using this metric for the construction of the embedding 
$\iota:{\cal R\cal D}(c)\to {\cal R\cal D}(c_1,c_2)$, we conclude that 
the intersections with
$B$ of the 
representatives $\hat \beta$ of 
the elements $\beta\in \iota({\cal R\cal D}(c))$ are pairwise disjoint. 

Define a map 
\[\Xi:{\cal D\cal G}\to \mathbb{Z}\] as 
follows. Let $\hat \beta$ be a closed piecewise geodesic
curve with properties (i),(ii) above which is constructed from the 
simple closed diskbounding curve $\beta$. 
Let $b$ be one of the components of 
$\hat \beta\cap B$ with endpoints on $\hat c_1$ and $\hat c_2$, 
oriented in such a way
that it connects $\hat c_1$ 
to $\hat c_2$. Such a component exists since otherwise the image of 
$\beta$ under the spot removal map is homotopic to a curve
disjoint from the diskbusting curve $c$ on $\partial H_0$. 
Let $x_1,x_2$ be the endpoints of $b$ on 
$\hat c_1,\hat c_2$.

Let $a=\alpha(x_1,x_2)$; then $b,a$ are simple arcs in $B$ with the
same endpoints which intersect some core curve of the annulus
$B$ in precisely one point. Assume that $\hat c_1,\hat c_2$ 
are oriented and define
the boundary orientation of $B$. Then $b$ is homotopic with fixed
endpoints to the arc  
$\hat c_1^k\cdot a \cdot \hat c_2^\ell$ for unique $k,\ell\in \mathbb{Z}$
(read from left to right).
In other words, if we denote by $\tau_i$ the positive
Dehn twist about $\hat c_i$, viewed as a diffeomorphism of the
punctured disk $B$ with fixed boundary, then $b$ is homotopic
with fixed endpoints 
to the arc $\tau_1^k\tau_2^{-\ell} a$. 
Define $\Xi(\beta)=k$. 

Observe that although this definition depends on the choice of the
arcs $\alpha(x_1,x_2)$ and on the choice of the component $b$
of $B\cap \hat \beta$, the map $\Xi$  
is coarsely well defined. 
Namely, let $b^\prime$ be a second component of 
$\hat\beta\cap B$, with endpoints $x_1^\prime,x_2^\prime$ 
on $\hat c_1,\hat c_2$ and distinct from $b$.
Then the interior of 
$b^\prime$ is disjoint from the interior of $b$.  
In particular, if $a^\prime$ is an arc in $B$ with the 
same endpoints as $b^\prime$ 
whose interior is
disjoint from $a$, then $b^\prime$ is homotopic with fixed endpoints
to $\tau_1^{q}\tau_2^{-r}a^\prime$
for $\vert q-k\vert \leq 1, \vert r-\ell\vert \leq 1$. On the other hand, 
the arcs $a=\alpha(x_1,x_2),\alpha(x_1^\prime,x_2^\prime)$ 
do not have an essential intersection with a fixed arc connecting 
$\hat c_1$ to $\hat c_2$ and hence
$a^\prime=\tau_1^s\tau_2^{-u}\alpha(x_1^\prime,x_2^\prime)$ for some
$\vert s\vert \leq 1,\vert u\vert \leq 1$.  
This shows that 
the multiplicity  
$k^\prime$ of the curve $\hat c_1$ in the description 
of $b^\prime$ relative to $\alpha(x_1^\prime,x_2^\prime)$ 
satisfies $\vert k-k^\prime\vert \leq 2$.
The same reasoning yields that the map 
$\Xi$ is coarsely two-Lipschitz. Furthermore, we have
$\Xi (\iota({\cal R\cal D}(c)))\subset [-2,2]$. Namely, recall 
that we chose the geodesic arc $\alpha$ in the beginning of this
proof to be contained in one of the curves $\iota({\cal R\cal D}(c))$
(which is nothing else but a normalization assumption). 

To summarize, the map
\[(\Theta,\Xi):{\cal D\cal G}\to {\cal R\cal D}(c)\times \mathbb{Z}\]
is coarsely Lipschitz, and its composition 
with the inverse of the map $\Psi$ from Lemma \ref{firststep}
is a coarse Lipschitz retraction 
of ${\cal D\cal G}$ onto $\Omega$ provided that 
the map $\Xi$ maps
a point in $\Lambda^k\iota({\cal R\cal D}(c))$ 
into a uniformly bounded neighborhood of 
$k$. 

However, if $\beta_0\in \iota({\cal R\cal D}(c))$ and if 
$\beta=\Lambda^k(\beta_0)\in 
\Lambda^k\iota({\cal R\cal D}(c))$, 
then the intersections with $H-{\rm int}(B)$ of the 
representatives $\hat \beta, \hat\beta_0$ 
of $\beta,\beta_0$ constructed above
coincide. This implies that up to homotopy with 
fixed endpoints, $\hat\beta\cap B=\Lambda^k(\hat\beta_0\cap B)$. 

On the other hand, point-pushing along a simple closed
curve $\gamma$ based at $p$ 
descends to conjugation by $\gamma$
in $\pi_1(\partial H_0,p)$. Therefore the image under the map $\Lambda$
of a simple 
arc $b$ in $B$ with endpoints on the two distinct components of 
$\partial B$ is homotopic with fixed endpoints to 
$c_1bc_2$ (recall that we oriented $c_1,c_2$ so that they
define the boundary orientation of $B$). 
As $\Xi(\iota({\cal R\cal D}(c)))\subset [-2,2]$, 
it follows that $\vert \Xi(\beta)-k\vert \leq 2$.
This shows the proposition.
\end{proof}

To summarize, we obtain

\begin{corollary}\label{nothyp}
The disk graph of a handlebody $H$ of genus $g\geq 2$ with one
spot contains quasi-isometrically embedded copies of $\mathbb{R}^2$.
\end{corollary}
\begin{proof}
A subgraph $\Gamma$ of a metric graph $G$ is uniformly quasi-isometrically 
embedded if there exists a coarsely Lipschitz retraction $G\to \Gamma$. Proposition 
\ref{calculate2} shows that for any $I$-bundle generator $c$ in $\partial H_0$, 
there is a coarse Lipschitz retraction 
of ${\cal D\cal G}$ onto its subgraph $\Omega=\cup_k\Lambda^k\iota({\cal R\cal D}(c))$, 
and by Lemma \ref{firststep}, 
$\Omega$ is quasi-isometric to the direct product 
${\cal R\cal D}(c)\times \mathbb{Z}$. Thus as 
by Lemma \ref{iso}, ${\cal R\cal D}(c)$ is quasi-isometric
to the arc graph of the base $F$ of the $I$-bundle determined by $c$ and hence has
infinite diameter, the product of any biinfinite geodesic in ${\cal R\cal D}(c)$ and 
$\mathbb{Z}$ defines a quasi-isometrically embedded $\mathbb{Z}^2$ in 
${\cal D\cal G}$.
\end{proof}

\begin{remark} In \cite{H19a} we showed that in contrast to handlebodies
without spots, the disk graph of 
a handlebody $H$ with a single spot on the boundary is \emph{not}
a quasi-convex subgraph of the curve graph of $\partial H$.
We do not know whether ${\cal D\cal G}$ contains quasi-isometrically embedded
euclidean spaces of dimension bigger than two. 
\end{remark}

\section{Once spotted doubled handlebodies}
 
In this section we consider the connected sum $M=\sharp_g S^2\times S^1$ 
of an even number $g=2n\geq 2$ of copies of $S^2\times S^1$ with one spot
(marked point).  We explain how the 
construction that led to the proof of Theorem \ref{diskgraph} 
can be used to show 
Theorem \ref{spheregraph}: The sphere graph of $M$ contains quasi-isometrically 
embedded copies of $\mathbb{R}^2$. 

Consider the double 
$M_0=\sharp_g S^2\times S^1$ of a handlebody $H_0$ of genus
$g\geq 2$ without spots. Let $M$ be the manifold $M_0$ equipped with
a marked point $p$. As before, we call $p$ a spot in $M$.
There is a natural spot removing map $\Phi:M\to M_0$.

The vertices of the \emph{sphere graph}  
${\cal S\cal G}$ of $M$
are isotopy classes of embedded spheres in $M$
which are disjoint from the spot and not isotopic into the spot. 
Isotopies are required to 
be disjoint from the spot as well. 
Two such spheres
are connected by an edge of length one if they can be realized disjointly.
%The automorphism group ${\rm Aut}(F_g)$ of the free group with
%$g$ generators acts on ${\cal S\cal G}$.
Similarly, let ${\cal S\cal G}_0$ be the sphere
graph of $M_0$.

Choose an embedded oriented surface $F_0\subset M_0$ of genus $n$ 
with connected boundary such that the inclusion 
$F_0\to M_0$ 
induces an isomorphism $\pi_1(F_0)\to \pi_1(M_0)$.
We may assume that the oriented $I$-bundle $H_0$ over $F_0$
is an embedded handlebody $H_0\subset M_0$ whose 
double equals $M_0$.
Thus every embedded essential arc $\alpha$ in $F_0$ with boundary in
$\partial F_0$ determines a sphere $\Upsilon_0(\alpha)$ 
in $M_0$ as follows. The interval bundle over $\alpha$ is an embedded 
essential disk in
$H_0$, with boundary in $\partial H_0$, and we let $\Upsilon_0(\alpha)$ be the double
of this disk. By construction, the sphere 
$\Upsilon_0(\alpha)$ intersects the surface 
$F_0$ precisely in the arc $\alpha$. 
By Lemma 4.17 of \cite{HH15}, distinct arcs give rise 
to non-isotopic spheres, furthermore the map $\Upsilon_0$
preserves disjointness and hence 
$\Upsilon_0$ is 
a simplicial embedding of the arc graph 
${\cal A}(F_0)$ of $F_0$ into the sphere graph ${\cal S\cal G}_0$ of 
$M_0$.

Now mark a point $p$ on the boundary $\partial F_0$ of $F_0$ and view
the resulting spotted surface $F$ as a surface in the spotted manifold $M$. 
The \emph{arc graph} ${\cal A}(F)$ of $F$ 
is the graph whose vertices are 
isotopy classes of essential simple arcs in $F$ with
endpoints on the complement of $p$ in the boundary of $F$.
Here we exclude arcs which are homotopic with fixed endpoints
to a subarc of $\partial F$ containing the base point $p$, and 
we require that an isotopy preserves the marked point $p$ and
hence endpoints of arcs can only slide along $\partial F-\{p\}$. 
Two such arcs are connected by an edge 
if they can be realized disjointly. 
Note that ${\cal A}(F)$ is \emph{not} the arc graph of the bordered 
surface $F$ punctured at an interior point of $F$.
Associate to an arc $\alpha$ in $F$ the double $\Upsilon(\alpha)$ of 
the $I$-bundle over $\alpha$. 

The spot removal map 
$\Phi:M\to M_0$ induces a simplicial surjection ${\cal S\cal G}\to {\cal S\cal G}_0$, 
again denoted by $\Phi$ for simplicity. Similarly, if we let 
$\phi:F\to F_0$ be the map which forgets the marked point $p\in \partial F$, then 
$\phi$ induces a simplicial surjection ${\cal A}(F)\to {\cal A}(F_0)$, denoted as well by
$\phi$. We then obtain a commutative diagram
\begin{equation}\label{diagram}
\begin{tikzcd}
{\cal A}(F)\arrow{r}{\phi} \arrow{d}{\Upsilon} & {\cal A}(F_0) \arrow{d}{\Upsilon_0}\\
{\cal S\cal G} \arrow{r}{\Phi} & {\cal S\cal G}_0
\end{tikzcd}
\end{equation}

Similar to the case of the handlebody $M_0$ without spots and the map
$\Upsilon_0$, we obtain

\begin{lemma}\label{injective}
The map $\Upsilon$  is a simplicial embedding of the
arc graph ${\cal A}(F)$ into the sphere graph. 
\end{lemma}
\begin{proof} We have to show that the map $\Upsilon$ is injective.
As $\Upsilon_0$ is injective and 
as the diagram (\ref{diagram}) commutes, it suffices to 
show the following. Let $\alpha\not=\beta\in {\cal A}(F)$ be such that
$\phi(\alpha)=\phi(\beta)$; then $\Upsilon(\alpha)\not=\Upsilon(\beta)$.

Now $\phi(\alpha)=\phi(\beta)$ means that up to exchanging $\alpha$ and 
$\beta$, there exists a number $k>0$ such that $\beta$ can be obtained from
$\alpha$ by $k$ half Dehn twists about the boundary $\partial F$ of $F$. 
Here the half Dehn twist $T(\alpha)$ of $\alpha$ is defined as follows. 

The orientation of 
$F$ induces a boundary orientation for $\partial F$ which in turn induces an 
orientation on $\partial F-\{p\}$. With respect to the
order defined by this 
orientation, let $x$ be the bigger of the two endpoints $x,y$ of $\alpha$. 
Slide $x$ across $p$ to obtain a new arc $T(\alpha)$, 
with endpoints $x^\prime,y$.
This arc is not homotopic to $\alpha$. To see this it suffices to show that 
the double $DT(\alpha)$ 
of $T(\alpha)$ in  the double $DF$ of $F$ 
(which is a surface with one puncture) is not freely
homotopic to the double $D(\alpha)$ of 
$\alpha$. This follows since 
$D(\alpha)$ and $DT(\alpha)$ can be homotoped in such a way that
they bound a once punctured annulus in $DF$.

The same reasoning also shows that the sphere $\Upsilon(T(\alpha))$ is not homotopic
to the sphere $\Upsilon(\alpha)$. Namely, 
let $\chi\subset \partial F\cup \{p\}$ be the oriented embedded arc connecting the intersection
point $x$ of $\alpha$ with $\partial F$ to the point $x^\prime$. 
This arc contains $p$ in its interior. Then 
the sphere $\Upsilon(T(\alpha))$ is a connected sum of the sphere
$\Upsilon(\alpha)$ with the boundary of a punctured ball which is a thickening of $\chi$.
Thus $\Upsilon(\alpha)$ and $\Upsilon(T(\alpha))$ can be isotoped in such a way
that they bound a subset of $M$ homeomorphic to the complement of an interior point of 
$S^2\times [0,1]$.

The above construction, applied to the sphere $\Upsilon(T(\alpha))$
instead of the sphere $\Upsilon(\alpha)$ and where the point $y$ takes on the role of
the point $x$ in the above discussion, 
shows that 
$\Upsilon(T^2(\alpha))$ is obtained
from $\Upsilon(\alpha)$ by point-pushing along the oriented 
loop $\partial F$ with basepoint $p$. This is a diffeomorphism of $M$ which leaves 
the complement of a small tubular neighborhood of $\partial F$ pointwise fixed and pushes
the basepoint $p$ along $\partial F$. 
As in the proof of Lemma \ref{firststep}, this argument
can be iterated. It shows that the sphere $\Upsilon(T^k(\alpha))$ intersects
the sphere $\Upsilon(\alpha)$ in $k-1$ intersection
circles. These circles are essential since they cut both $\Upsilon(T^k(\alpha))$ and 
$\Upsilon(\alpha)$ into
two disks and $k-2$ annuli, where a disk component of $T^k(\alpha)-T(\alpha)$ 
bounds together with a disk component of $T(\alpha)-T^k(\alpha)$ an embedded
sphere enclosing the spot. 
Invoking the proof of 
Lemma \ref{firststep}, we conclude that    
indeed, for $k\not=\ell$, 
$\Upsilon(T^k(\alpha))$ is not homotopic to $\Upsilon(T^\ell(\alpha))$. 

We showed so far that the map $\Upsilon$ is injective. To complete the proof of
the lemma, it suffices to observe that disjoint arcs are mapped to disjoint 
spheres. But this is immediate from the construction.
\end{proof}

Proposition 4.18 of \cite{HH15} shows that there is a one-Lipschitz
retraction 
\[\Psi_0:{\cal S\cal G}_0\to\Upsilon_0({\cal A}(F_0))\]
which is of the form $\Psi_0=\Upsilon_0\circ \Theta_0$ 
(read from right to left) 
where $\Theta_0:{\cal S\cal G}_0\to {\cal A}(F_0)$ 
is a one-Lipschitz map. 
In particular, $\Upsilon_0({\cal A}(F_0))$ is a quasi-isometrically embedded 
subgraph of ${\cal S\cal G}_0$ which is quasi-isometric to ${\cal A}(F_0)$. 
Our goal is to show that there also is a 
coarse Lipschitz retraction
of ${\cal S\cal G}$ onto $\Upsilon({\cal A}(F))$ 
of the form $\Psi=\Theta\circ \Upsilon$ 
where $\Theta:{\cal S\cal G}\to {\cal A}(F)$ is a coarse Lipschitz map. 
This then yields 
Theorem \ref{spheregraph} from the introduction.

To construct the map $\Theta$ we use the method
from \cite{HH15}.  We next explain how this 
method can be adapted to our needs.

Let as before $F\subset M$ be an embedded oriented surface with 
connected boundary
$\partial F$ so that $M$ is 
the double of the trivial $I$-bundle over $F$. 
We assume that the marked point $p$ is contained in the
boundary $\partial F$ of $F$. Furthermore, we 
assume that  the boundary $\partial F$ of $F$ is a smoothly  embedded 
circle in $M\cup \{p\}$ (that is, an embedded compact one-dimensional submanifold). 
We use the marked point $p$ as the basepoint for the fundamental group
of $M$. Then $\partial F$ equipped with its boundary orientation defines a homotopy
class $\beta\in \pi_1(M,p)=\pi_1(F,p)={\cal F}_{2g}$ (the free group in
$2g$ generators). Since $\beta$ is the oriented boundary curve of $F$, it is an iterated
commutator in a standard set of generators of ${\cal F}_{2g}$ and hence 
$\beta$ is not contained in any free factor (Whitehead graphs are a convenient tool to 
verify this fact). Thus 
$\partial F$ intersects every sphere in $M$. Namely, for any given sphere $S$ in $M$,
the subgroup of $\pi_1(M,p)$ of all homotopy classes of loops which 
do not intersect $S$ is a proper free factor of $\pi_1(M,p)$.

As in \cite{HH15} and similar to the construction in Lemma \ref{casewithoutspot}, 
the strategy is to associate to a sphere $S$ in $M$ a component of the intersection 
$F\cap S$. However, unlike in the case of curves on surfaces, there is no suitable normal form for
intersections of spheres with the surface $F$, 
and the main work in \cite{HH15} consists of overcoming this difficulty
by introducing a relative normal form which allows to associate to a sphere in $M_0$ an 
intersection arc with $F_0$ so that the resulting map ${\cal S\cal G}_0\to {\cal A}(F_0)$ 
is one-Lipschitz.

For the remainder of this section we outline the main steps in this construction, adapted
to the sphere graph ${\cal S\cal G}$ of $M$ and the arc graph ${\cal A}(F)$ of $F$. 
This requires modifying spheres with isotopies
not crossing through $p$,  
and modifying the surface $F$ with 
homotopies leaving the boundary $\partial F$ pointwise fixed. 

For convenience, we record some definitions from \cite{HH15} (the following combines Definition 4.7 
and Definition 4.9 of \cite{HH15}). 

\begin{definition}\label{minimal}
%Let $\gamma$ be an embedded closed curve $M$ containing the basepoint $p$ and 
Let $\Sigma$
be a sphere or a sphere system. 
\begin{enumerate}
\item  $\partial F$ \emph{intersects $\Sigma$ minimally} if $\partial F$ intersects $\Sigma$ transversely
and if no component of $\partial F-\Sigma$ not containing the basepoint $p$ is homotopic with fixed
endpoints into $\Sigma$.
\item $F$ \emph{is in minimal position with respect to $\Sigma$} if $\partial F$ intersects $\Sigma$ minimally and 
  if moreover each component of $\Sigma\cap F$ is a properly embedded arc
  which either is essential or homotopic with fixed endpoints
 to a subarc of $\partial F$ containing the marked point. 
\end{enumerate}
\end{definition}

A version of the easy Lemma 4.6 of \cite{HH15} states that any closed curve containing the basepoint 
can be put into minimal position relative to a sphere system $\Sigma$
as defined in the first part of Definition \ref{minimal}.
The following is a version of Lemma 4.12 of \cite{HH15}. For its formulation, call a sphere
system $\Sigma$ \emph{simple} if it decomposes $M$ into a simply connected components.

\begin{lemma}\label{minimal2}
Let $\Sigma$ be a simple sphere system in $M$. 
Suppose that $F$ is in minimal position with respect to $\Sigma$.
Let $\sigma^\prime$ be an embedded sphere disjoint from $\Sigma$ and let $\Sigma^\prime$
be a simple sphere system obtained from $\Sigma$ by either adding $\sigma^\prime$, or removing
one sphere $\sigma\in \Sigma$. Then $F$ can be homotoped leaving $p$ fixed to a surface $F^\prime$
which is in minimal position with respect to $\Sigma^\prime$. 
\end{lemma}
\begin{proof}
As in the proof of Lemma 4.12 of \cite{HH15}, removing a sphere preserves minimal position, so only
the case of adding a sphere has to be considered. 

Thus let $\Sigma$ be a simple sphere system and let $\sigma^\prime$
be a sphere disjoint from $\Sigma$.
Assume that $F$ is in minimal position with respect to $\Sigma$. Let $W_\Sigma$ be the complement of 
$\Sigma$ in $M$, that is, $W_\Sigma$ is a compact (possibly disconnected) manifold whose boundary
consists of $2k$ boundary spheres $\sigma_1^+,\sigma_1^-,\cdots,\sigma_k^+,\sigma_k^-$. The boundary
spheres $\sigma_i^+$ and $\sigma_i^-$ correspond to the two sides of a sphere $\sigma_i\in \Sigma$. 
The surface $F$ intersects $W_\Sigma$ in a collection of embedded surfaces with boundaries. 
Each such surface is a polygonal disk $P_i$ $(i=1,\dots,m)$. The sides of each such polygon 
alternate between subarcs of $\partial F$ and arcs contained in $\Sigma$. There is at most one bigon,
that is, a polygon with two sides, and this polygon then contains the point $p$ in one of its sides. 
Each rectangle, if any, is homotopic 
into $\partial F$. 

The proof of Lemma 4.12 of \cite{HH15} now proceeds by studying the intersection of each polygonal 
component of $F-\Sigma$ with the sphere $\sigma^\prime$. This is done by contracting
each such polygonal
component $P$ to a ribbon tree $T(P)$ 
in such a way that the boundary components in $\Sigma$ are contracted to 
single points in $T(P)$. If $P$ is not a rectangle or bigon, then
$T(P)$ has a single vertex which is not univalent. As such ribbon trees are 
one-dimensional objects, they can be homotoped with fixed endpoints on $\partial W_\Sigma$ to trees
which are in minimal position with respect to $\sigma^\prime$. 
This construction applies without change to rectangles
and perhaps the bigon which 
can be represented by
an interval with one endpoint at $p$ and the second endpoint on a component of $\Sigma$.  
We refer to the proof of Lemma 4.12 of \cite{HH15} for details. No adjustment of the argument is 
necessary. 
 \end{proof}

 The above construction is only valid for simple sphere systems $\Sigma$
 and not for individual spheres. Furthermore, it is known that 
the arc system on $F\cap \Sigma$ obtained by putting $F$ into 
minimal position with respect to $\Sigma$ is not uniquely determined
by $\Sigma$.  To overcome this difficulty, the work of \cite{HH15} uses as an auxiliary datum 
a maximal system $A_0$ of pairwise disjoint essential 
arcs on the surface $F_0$. Here maximal means
that any arc which is disjoint from $A_0$ is contained in 
$A_0$. The system $A_0$ then  \emph{binds} $F_0$, that is,
$F-A_0$ is a union of topological disks. 
Furthermore, $\partial F_0$  and each arc $\alpha\in A_0$ is equipped with an orientation. 

Choose an arc system $A$ for $F$ which binds $F$.
If $F\subset M$ is in minimal position with respect to $\Sigma$, then a homotopy assures that 
no arc from the arc system $A$ intersects a component of $F-\Sigma$ which is a rectangle or 
a bigon. Then Lemma 4.12 of \cite{HH15} and its proof applies without modification and shows that
with a homotopy, $F$ can be put into normal form
with respect to the arc system $A$, called
\emph{$A$-tight minimal position} with respect to $\Sigma$. 
This then yields the statement of Lemma 4.16 of \cite{HH15}:
if $F$ is in $A$-tight minimal position with respect to the simple sphere system $\Sigma$, 
then the binding arc system $\Sigma \cap F$ is determined by 
$\Sigma$. In particular, two distinct spheres from $\Sigma$ intersect $F$ in disjoint essential arcs.
There may in addition be inessential arcs, that is, arcs which are homotopic with
fixed endpoints to a subsegment of $\partial F$ containing the basepoint $p$, 
but these will be unimportant for our purpose.

Now let $\sigma$ be 
an essential sphere 
in $M$. Let $\Sigma$ be a simple sphere system in $M$ 
containing $\sigma$ as a component. 
We put $F$ into $A$-tight minimal position with respect to $\Sigma$. 
Then $\sigma\cap F$ consists of a non-empty collection of essential arcs
and perhaps some additional non-essential arcs. 
Choose one of the essential 
intersection arcs $\alpha$ and define $\Theta(\sigma)=\alpha$. 
As in \cite{HH15} and Proposition \ref{calculate2} we now obtain

\begin{proposition}\label{next1}
The map $\Theta$ is a coarsely Lipschitz map. For each arc $\alpha\in {\cal A}(F)$, 
we have $\Theta(\Upsilon(\alpha))=\alpha$. 
As a consequence, 
if $g=2n$ is even then the sphere graph ${\cal S\cal G}$ 
of $M$ contains
quasi-isometrically embedded copies of $\mathbb{R}^2$.
\end{proposition}
\begin{proof} Given the above discussion, the proof that 
$\Theta$ is a coarsely Lipschitz map is identical to the proof that the map $\Theta_0$ is 
a coarsely Lipschitz map in Proposition 4.18 of \cite{HH15} and will be omitted. Moreover, as for $\alpha
\in {\cal A}(F)$, the sphere 
$\Upsilon(\alpha)$ intersects $F$ in the unique arc $\alpha$, we have 
$\Theta(\Upsilon(\alpha))=\alpha$.

As a consequence, $\Theta\vert \Upsilon({\cal A}(F))$ is a Lipschitz bijection, 
with inverse $\Upsilon$. Then 
the subgraph $\Upsilon({\cal A}(F))$ of ${\cal S\cal G}$ is bilipschitz equivalent
to ${\cal A}(F)$. Furthermore, the map $\Upsilon\circ \Theta$ 
is a Lipschitz retraction of ${\cal S\cal G}$ onto $\Upsilon({\cal A}(F))$. 
Then $\Upsilon({\cal A}(F))$ is a quasi-isometrically embedded subgraph of
${\cal S\cal G}$ which is moreover quasi-isometric to ${\cal A}(F)$.

Let as before $F_0$ be the surface obtained from $F$ by removing the spot. 
We are left with showing that ${\cal A}(F)$ is quasi-isometric to 
${\cal A}(F_0)\times \mathbb{Z}$. However, this was shown in Lemma \ref{firststep}.
Namely, in the terminology used before, the boundary $\partial F$ is 
an $I$-bundle generator in the trivial interval bundle $H$ over $F$, and 
associating to an arc $\alpha$ the $I$-bundle over $\alpha$ defines an isomorphism
of ${\cal A}(F)$ with the subgraph $\Omega$ of the disk graph of $H$ used in 
Lemma \ref{firststep}. The statement now follows from Lemma \ref{firststep}.
\end{proof}

\begin{remark}
 Most likely Proposition \ref{next1} holds true as well in
the case that $g=2n+1$ is odd, and furthermore this 
can  be deduced with the above argument 
using non-orientable surfaces. However, 
the analogue of Proposition 4.18 of 
\cite{HH15} for non-orientable surfaces is not available, and we
leave the verification of these claims to other authors.
\end{remark}

\bigskip\bigskip

\noindent
MATH. INSTITUT DER UNIVERSIT\"AT BONN, ENDENICHER ALLEE 60, 
53115 BONN, GERMANY\\
\bigskip\noindent
e-mail: ursula@math.uni-bonn.de

\end{document}